\documentclass[10pt,reqno,a4paper]{amsart}%
\usepackage{amsfonts,amsmath,times}
\usepackage{tikz}
\usepackage{shuffle}
\usepackage{enumitem}

\usepackage{academicons}

\usepackage[sc]{mathpazo} 
\usepackage[scaled]{helvet} 
\usepackage{eulervm} 

\definecolor{grau}{rgb}{0.65,0.65,0.65}
\definecolor{dblau}{rgb}{0,0,0.45}
\definecolor{blau}{rgb}{0,0,0.75} 
\definecolor{grun}{rgb}{0.1,0.6,0.1} 
\usepackage[pdftex]{hyperref}
\hypersetup{colorlinks,linkcolor=grun,citecolor=blue,urlcolor=blau}

\newcommand{\myt}[1]{#1}

\theoremstyle{plain}
\newtheorem{lem}{\normalfont\scshape Lemma}
\newtheorem{thm}{\normalfont\scshape Theorem}
\newtheorem{coroll}{\normalfont\scshape Corollary}

\theoremstyle{definition}
\newtheorem{remark}{\normalfont\scshape Remark}
\newtheorem{example}{\normalfont\scshape Example}
\newtheorem{defi}{\normalfont\scshape Definition}

\def\vt{\vec{t}}
\def\vu{\vec{u}}
\def\Ztt{Z^{t}}
\def\ztt{\zeta^{t}}
\def\zeto{\zeta^{t_E,t_O}}
\def\Ztm{Z^{\vec{t}}}
\def\ztm{\zeta^{\vec{t}}}

\def\mstuff{\ensuremath{\stackrel{\vec{t}}\ast}}
\def\istuff{\ensuremath{\stackrel{t}\ast}}

\newcommand{\stuff}[1]{\stackrel{#1}{\ast}}
\newcommand{\Cont}[1]{\text{Con}_{#1}^{\vec{t}}}
\newcommand{\Con}[1]{\text{Con}_{#1}}
\newcommand{\hneu}{\ensuremath{\mathfrak{h}^2}}

\def\vsig{\vec{\sigma}}

\def\zt{\zeta}
\def\zts{\zeta^{\star}}

\def\R{\mathbf R}

\newcommand{\M}{\ensuremath{\mathcal{M}}}

\newcommand{\N}{\ensuremath{\mathbb{N}}}

\newcommand{\orcid}[1]{\href{https://orcid.org/#1}{\includegraphics[scale=0.35]{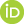}}}

\begin{document}
\title{On multi-interpolated multiple zeta values}%
\author[M.~Kuba]{Markus Kuba }
\address{Markus Kuba \orcid{0000-0001-7188-6601}\\
Department Applied Mathematics and Physics\\
FH - Technikum Wien\\
H\"ochst\"adtplatz 6\\
1200 Wien, Austria} %
\email{kuba@technikum-wien.at}
\date{\today}

\begin{abstract}
In this note we introduce multi-interpolated multiple zeta values. We provide a basic decomposition of these objects involving ordered partitions.
We also obtain identities for special instances of multi-interpolated multiple zeta values $\zeta^{\vec{t}}(\{s\}_k)$, generalizing earlier results. 
Moreover, we introduce a product for multi\--interpolated multiple zeta values. 
\end{abstract}


\maketitle


\section{Introduction}

The \myt{multiple zeta values}, henceforth MZVs, are defined by
\[
\zt(i_1,\dots,i_k)=\sum_{\ell_1>\cdots>\ell_k\ge 1}\frac1{\ell_1^{i_1}\cdots \ell_k^{i_k}},
\]
with admissible indices $(i_1,\dots,i_k)\in\N^k$ satisfying $i_1\ge 2$, $i_j\ge 1$ for $2\le j\le k$, see Hoffman~\cite{H92} and Zagier~\cite{Z}. 
We refer to $i_1 +\dots + i_k$ as the weight of this MZV, and $k$ as its depth. 
For a comprehensive overview as well as a great many pointers to the literature we refer to the survey of Zudilin~\cite{Zu}.
An important variant of the MZVs are the so-called multiple zeta \myt{star values}, abbreviated by MZSVs, where equality is allowed:
\[
\begin{split}
  \zts(i_1,\dots,i_k) & = \sum_{\ell_1\ge \cdots\ge \ell_k\ge 1}\frac1{\ell_1^{i_1}\cdots \ell_k^{i_k}}.
\end{split}
\]
Yamamoto~\cite{Y} introduced a generalization $\ztt$ called \myt{interpolated multiple zeta values} of both $\zt$ and $\zts$. Noting that, 
\[
\zts(i_1,\dots,i_k)=\sum_{\circ = \text{``},\text{''} \text{or} \, \text{``}+\text{''}}\zt(i_1\circ i_2 \dots \circ i_k),
\]
let the parameter $\sigma$ denote the number of plus in the expression $i_1\circ i_2 \dots \circ i_k$. Interpolated multiple zeta values are defined by
\begin{equation}
\label{EqYamamoto}
  \ztt(i_1,\dots,i_k) = \sum_{\circ = \text{``},\text{''} \text{or} \, \text{``}+\text{''}}t^{\sigma}\zt(i_1\circ i_2 \dots \circ i_k).
\end{equation}
Thus, the series $\ztt(i_1,\dots,i_k)$ is a polynomial in $t$ and interpolates between MZVs, $\zeta^0=\zeta$, and MZSVs, $\zeta^1=\zts$.
Equivalently, the interpolated MZVs can be defined as
\begin{equation}
\label{EqYamamoto2}
  \ztt(i_1,\dots,i_k) =  \sum_{\ell_1\ge \cdots\ge \ell_k\ge 1}\frac{t^{\sigma(\ell_1,\dots,\ell_k)}}{\ell_1^{i_1}\cdots \ell_k^{i_k}},
\end{equation}
where $\sigma$ is given by the number of equalities:
\begin{equation}
\label{DefSigma}
\sigma(\ell_1,\dots,\ell_k)=|\{1\le r\le k-1\mid  \ell_r=\ell_{r+1}\}|.
\end{equation}
It turned out that the interpolated series satisfies 
many identities generalizing or unifying earlier result for multiple zeta and zeta star values, see for example Yamamoto~\cite{Y} for a generalization of the sum identity as well as many other results, and also Hoffman and Ihara~\cite{HI} or Hoffman~\cite{H2000,H2018} for further results. We also refer to Tanaka and Wakabayashi~\cite{TaWa} for a proof of Kawashima’s relations, 
to Wakabayashi~\cite{Waka} as well as Li and Qin~\cite{LiQin1} for Double shuffle and Hoffman’s relations, 
and to Li~\cite{Li} for a recent study of algebraic aspects, including extended double shuffle relations, symmetric sum formulas and restricted sum formulas. 

\smallskip

Most important, the so-called $t$-harmonic product $\istuff$ was introduced by Yamamoto~\cite{Y} for the interpolated MZVs, see also~\cite{H2000,H2018,HI}. It satisfies
\begin{equation}
\label{eqn:YamamotoIdentity}
\ztt\big( (i_1,\dots,i_k)\istuff (j_1,\dots,j_\ell)\big)=\ztt(i_1,\dots,i_k)\cdot \ztt(j_1,\dots,j_\ell),
\end{equation}
where the product $\istuff$ is actually defined in an algebraic way (see Section~\ref{SecAlg}). Hoffman and Ihara used a general algebra framework, leading, amongst others to expressions for interpolated MZVs $\ztt(\{s\}_{k})$ in terms of Bell polynomials
 and ordinary single argument zeta values, $m\ge 1$. Here and throughout this work $\{s\}_{k}$ means $s$ repeated $k$ times.


\smallskip

In this work we introduce a generalization $\vsig$ of the parameter $\sigma$~\eqref{DefSigma}, generalizing interpolated MZVs~\eqref{EqYamamoto2}
to what we call multi-interpolated MZVs. This allows to gain more insight into structural decompositions of ordinary interpolated MZVs, as well as a link to the $\mathfrak{t}$-values of Hoffman~\cite{H2019}.
Our results are the following. First, we obtain in Theorem~\ref{the:decomposition} a decomposition of multi\--interpolated multiple zeta values using ordered partitions.
Second, we obtain several identities for multi-interpolated multiple zeta values, $\zt^{\vec{t}}(\{s\}_k)$, see Theorem~\ref{TheGF} and Corollary~\ref{Co1}, generalizing earlier results for $\ztt(\{s\}_{k})$.
Third, as our main result we introduce in Section~\ref{SecAlg} a product $\mstuff$ for multi-interpolated multiple zeta values, generalizing~\eqref{eqn:YamamotoIdentity}.
Interestingly, it turns out that the product $\mstuff$ involves a non-commutative variable $\vec{t}$, 
in contrast to the earlier $t$-harmonic product $\istuff$ for interpolated MZVs.

\section{Multi-interpolated multiple zeta values and related multiple zeta values}
\begin{defi}[Multi-interpolated multiple zeta values]
\label{DefMuli}
Given integers $(i_1,\dots,i_k)$ with $i_1\ge 2$, $k\ge 1$ and a sequence of variables $\vec{t}=(t_1,t_2,\dots)$. 
For our purpose we assume that $t_j\in[-1;1]$, $j\in\N$. 
The multi-interpolated multiple zeta value $ \ztm(i_1,\dots,i_k)$ is defined by
\begin{equation}
\label{def:ztm}
\ztm(i_1,\dots,i_k)=\sum_{\ell_1\ge \cdots\ge \ell_k\ge 1}\frac{\vec{t}^{\vec{\sigma}(\vec{\ell})}}{\ell_1^{i_1}\dots \ell_k^{i_k}}
= \sum_{\ell_1\ge \cdots\ge \ell_k\ge 1}\frac{\prod_{j=1}^{\infty}
t_{j}^{\sigma_j(\ell_1,\dots,\ell_k)}}{\ell_1^{i_1}\dots \ell_k^{i_k}},
\end{equation}
with parameter $\sigma_j$ denoting the number of equalities of the integer $j\in\N$:
\begin{equation}
\label{def:sigmaj}
\sigma_j(\ell_1,\dots,\ell_k)=|\{1\le r\le k-1\mid  \ell_r=\ell_{r+1}=j\}|.
\end{equation}
\end{defi}

\begin{remark}
\label{rem:sigma}
Note that the parameters $\sigma_j$ refine the parameter $\sigma$~\eqref{DefSigma} due to the identity
\[
\sigma = \sum_{j\ge 1}\sigma_j.
\] 
Hence, by setting $t_j=t$, $j\ge 1$, which we write in a slight abuse of notation 
simply as $\vec{t}=t$, we have 
\begin{equation}
\label{eqn:tgleicht}
\prod_{j=1}^{\infty}t_{j}^{\sigma_j(\vec{\ell})}=t^{\sum_{j=1}^{\infty} \sigma_j(\vec{\ell})}=t^{\sigma(\vec{\ell})},
\end{equation}
so that $\ztm(i_1,\dots,i_k)$ reduces to $\ztt(i_1,\dots,i_k)$. 
\end{remark}
\begin{example}[Multi-interpolated MZVs: depth one]
\label{ExDeco1}
For depth one, the multi-interpolated MZVs reduce to ordinary zeta values:
\[
\ztm(i)=\ztt(i)=\zt(i),\quad i>1.
\]
\end{example}

\begin{example}[Multi-interpolated MZVs: depth two]
\label{ExDeco2}
The set $\{\ell_1\ge\ell_2\ge1 \}$ is split into two parts,
\[
\{\ell_1\ge\ell_2\ge1 \}=\{\ell_1>\ell_2\ge1 \}\cup\{\ell_1=\ell_2\ge1 \},
\]
leading to
\begin{align}
\label{ExDeco2eqn1}
\ztm(i_1,i_2)&=\sum_{\ell_1\ge  \ell_2\ge 1}\frac{\vec{t}^{\vec{\sigma}(\vec{\ell})}}{\ell_1^{i_1}\ell_2^{i_2}}
=\sum_{\ell_1> \ell_2\ge 1}\frac{1}{\ell_1^{i_1}\ell_2^{i_2}}
+ \sum_{\ell\ge 1}\frac{t_{\ell}}{\ell^{i_1+i_2}}=\zt(i_1,i_2)+\sum_{\ell\ge 1}\frac{t_{\ell}}{\ell^{i_1+i_2}}.
\end{align}
Here we used $\vec{t}^{\vec{\sigma}(\vec{\ell})}=1$ for $\vec{\ell}\in\{\ell_1>\ell_2\ge1 \}$ and $\vec{t}^{\vec{\sigma}(\vec{\ell})}=t_{\ell}$ for $\vec{\ell}\in\{\ell=\ell_1=\ell_2\ge1 \}$.
 In the special case $\vec{t}=t$~\eqref{eqn:tgleicht} we reobtain the ordinary interpolated MZVs and the evaluation
\[
\ztt(i_1,i_2)=\zt(i_1,i_2)+t\cdot \zt(i_1+i_2).
\]
\end{example}

In our previous example, different basic objects like $\sum_{\ell\ge 1}\frac{t_{\ell}}{\ell^{i_1+i_2}}$ appeared, compared to the ordinary interpolated MZVs. Thus, in order to analyze multi-interpolated MZVs, we need
another generalization, which includes interpolated MZVs (thus, also ordinary MZVs and MZSVs), as well as the atomic parts of the multi-interpolated truncated MZVs. 
\begin{defi}
\label{DefNewZeta1}
For $k\ge 1$ let $j_1,\dots,j_k\ge 0$ denote integers. We introduce multiple zeta values with variables $\vec{t}=(t_1,t_2,\dots)$
and admissible indices $(i_1,\dots,i_k)$ satisfying $i_1\ge 2$, $i_j\ge 1$ for $2\le j\le k$, 
\begin{equation*}
\zt\big((\vt^{j_1},i_1),\dots,(\vt^{j_k},i_k)\big)=\sum_{\ell_1>\cdots>\ell_k\ge 1}\frac{t_{\ell_1}^{j_1}\cdots t_{\ell_k}^{j_k}}{\ell_1^{i_1}\cdots \ell_k^{i_k}}.
\end{equation*}
\end{defi}

\begin{remark}[Connection to multiple $\mathfrak{t}$-values]
The definition above also covers the multiple $\mathfrak{t}$-values of Hoffman~\cite{H2019}.
The $\mathfrak{t}$-values~\cite{N} and multiple $\mathfrak{t}$-values~\cite{H2019} are defined by
\begin{equation}
\label{def:multipleT}
\mathfrak{t}(i_1,\dots,i_k)=\sum_{\substack{\ell_1>\cdots>\ell_k\ge 1\\\ell_i \text{odd}}}\frac1{\ell_1^{i_1}\cdots \ell_k^{i_k}}.
\end{equation}
We note in passing that $\mathfrak{t}(i)=(1-2^{-i})\zt(i)$.
Setting $t_\ell=(1-(-1)^\ell)/2$, $\ell\ge 1$, and $j_1=\dots=j_k=1$ leads to
\[
\zt\big((\vt,i_1),\dots,(\vt,i_k)\big)=\mathfrak{t}(i_1,\dots,i_k).
\]
\end{remark}

Note that for $j_\ell=0$, $1\le \ell\le k$, or $t_\ell=1$, $\ell\ge 1$, we simply write $i_\ell$ instead of $(\vec{1},i_\ell)$.
The MZSVs with variables $\vec{t}=(t_1,t_2,\dots)$ and also the multi-interpolated generalizations
$\ztm\big((\vt^{j_1},i_1),\dots,(\vt^{j_k},i_k)\big)$ are defined accordingly. 
For $t_{\ell_1}=-1$, $\ell_1\ge k$, the value $j_1=1$ is also admissible; this leads to alternating MZVs.
The generalized MZVs in Definition~\ref{DefNewZeta1} can be generalized further by setting $\vu_m=(u_{m,1},u_{m,2},\dots)$ and we get
\begin{equation}
\label{DefNewZeta3}
\zt\big( (\vu_1,i_1),\dots,(\vu_k,i_k)\big)=\sum_{\ell_1>\cdots>\ell_k\ge 1}\frac{u_{1,\ell_1}\cdots u_{k,\ell_k}}{\ell_1^{i_1}\cdots \ell_k^{i_k}},
\end{equation}
such that for $\vu_m=\vt^{j_m}$, $1\le m\le k$ we reobtain our earlier definition, but for and $\vu_m=(x_m^n)_{n\ge 1}$ we obtain multiple polylogarithms. Moreover, mixture models of multiple $\mathfrak{t}$-values and zeta values can be obtained by suitable choices of $\vu_m$.

\begin{example}[Multi-interpolated MZVs: depth three]
\label{ExDeco3}
We decompose $\ztm(i_1,i_2,i_3)$ into summands by splitting the underlying set into four parts:
\begin{align*}
\{\ell_1\ge\ell_2\ge \ell_3\ge 1\}&=\{\ell_1>\ell_2> \ell_3\ge 1\}\cup
\{\ell_1>\ell_2= \ell_3\ge 1\}\\
&\quad\cup \{\ell_1=\ell_2> \ell_3\ge 1\}\cup \{\ell_1=\ell_2= \ell_3\ge 1\},
\end{align*}
such that
\begin{align*}
\ztm(i_1,i_2,i_3)&=\sum_{\ell_1\ge\ell_2\ge \ell_3\ge 1}\frac{\vec{t}^{\vec{\sigma}(\vec{\ell})}}{\ell_1^{i_1}\ell_2^{i_2}\ell_3^{i_3}}\\
&=\zt(i_1,i_2,i_3)+\zt(i_1,(\vt,i_2+i_3))+\zt((\vt,i_1+i_2),i_3)+\zt((\vt^2,i_1+i_2+i_3)).
\end{align*}
\end{example}

A natural specialization of $\ztm(i_1,\dots,i_k)$ are even-odd interpolations. 

\begin{example}[Even-odd interpolated multiple zeta values]
\label{EvenOdd}
Given the multi-interpolated MZV $\ztm(i_1,\dots,i_k)$, we choose $t_{2m}=t_E$ and $t_{2m-1}=t_{O}$, $m\ge 1$, obtaining the even-odd interpolation 
\[
\zeto(i_1,\dots,i_k)=\sum_{\ell_1\ge \cdots\ge \ell_k\ge 1}\frac{t_E^{\sigma_E(\ell_1,\dots,\ell_k)}\cdot t_O^{\sigma_O(\ell_1,\dots,\ell_k)}}{\ell_1^{i_1}\cdots \ell_k^{i_k}},
\]
where $\sigma_E$ and $\sigma_O$ are given by the number of even and odd equalities~\eqref{def:sigmaj}, respectively.
Note that here, variants of the multiple $\mathfrak{t}$-values of Hoffman~\cite{H2019}, see~\eqref{def:multipleT},
naturally appear, as well as mixtures of multiple zeta and $\mathfrak{t}$-values (or multiple Hurwitz-zeta values); for example, 
in the case of depth two we obtain from Example~\ref{ExDeco2},~\eqref{ExDeco2eqn1} the decomposition
\begin{align*}
\ztm(i_1,i_2)
&=\zt(i_1,i_2)+t_E\cdot \frac{1}{2^{i_1+i_2}}\zt(i_1+i_2) + t_O\cdot \mathfrak{t}(i_1+i_2).
\end{align*}
\end{example}

\subsection{Ordered partitions and multi-interpolated MZVs}

Motivated by the special cases of depth one, two and three in Examples~\ref{ExDeco1},~\ref{ExDeco2} and~\ref{ExDeco3}, and also the very definition of interpolated MZVs~\eqref{EqYamamoto}, we provide a representation of multi-interpolated MZVs in terms of the values in Definition~\ref{DefNewZeta1}.
We start from the representation of interpolated multiple zeta values by ordered partitions $\mathcal{P}(k)$ of the integer $k$, also called compositions. We associate to each ordered partition $\mathbf{p}=(p_1,\dots,p_r)\in\mathcal{P}(k)$ a map from $\N^k$ to $\N^{\mathcal{L}(\mathbf{p})}$, where $\mathcal{L}(.)$ denote the length of the ordered partition:
\[
\mathbf{p}(i_1,\dots,i_k)=(\sum_{j_1=1}^{P_1}i_{j_1},\sum_{j_1=P_1+1}^{P_2}i_{j_1} \dots, \sum_{j_r=P_{r-1}+1}^{P_r}i_{j_r}).
\]
Here we use the convention $P_j=\sum_{\ell=1}^{j}p_\ell$. Then, 
\[
\ztt(i_1,\dots,i_k)=\sum_{\mathbf{p}\in\mathcal{P}(k)}t^{k-\mathcal{L}(\mathbf{p})}\zt(\mathbf{p}(i_1,\dots,i_k)),
\]
as each $\mathbf{p}=(p_1,\dots,p_r)\in\mathcal{P}(k)$ corresponds to a unique partition of the set 
\[
M=\{\ell_1\ge \dots\ge \ell_k\ge1\}=\bigcup_{\mathbf{p}\in\mathcal{P}(k)}\mathbf{p}(M)
\]
into subsets
\[
\mathbf{p}(M)=\{\ell_1\ge \dots\ge \ell_k\ge 1\colon \ell_1=\dots=\ell_{P_1}>\dots> \ell_{P_{r-1}+1}=\dots=\ell_{P_r}\}.
\]
Consequently, we obtain the following representation.
\begin{thm}
\label{the:decomposition}
The multi-interpolated multiple zeta values $\ztm(i_1,\dots,i_k)$, with $k\in\N$ and arguments $i_1>1$, $i_2,\dots,i_k\in\N$,
can be expressed in terms of ordered partitions $\mathbf{p}\in\mathcal{P}(k)$:
\begin{align*}
\ztm(i_1,\dots,i_k)=\sum_{\mathbf{p}\in\mathcal{P}(k)}\zt\Big((\vec{t}^{p_1-1},\sum_{j_1=1}^{P_1}i_{j_1}) \dots, (\vec{t}^{p_r-1},\sum_{j_r=P_{r-1}+1}^{P_r}i_{j_r})\Big).
\end{align*}
\end{thm}
\begin{proof}
\begin{align*}
\ztm(i_1,\dots,i_k)&=\sum_{\ell_1\ge \cdots\ge \ell_k\ge 1}\frac{\vec{t}^{\vec{\sigma}(\vec{\ell})}}{\ell_1^{i_1}\dots \ell_k^{i_k}}
=\sum_{(\ell_1,\dots,\ell_k)\in M}\frac{\vec{t}^{\vec{\sigma}(\vec{\ell})}}{\ell_1^{i_1}\dots \ell_k^{i_k}}\\
&=\sum_{\mathbf{p}\in\mathcal{P}(k)}\sum_{(\ell_1,\dots,\ell_k)\in \mathbf{p}(M)}\frac{\vec{t}^{\vec{\sigma}(\vec{\ell})}}{\ell_1^{i_1}\dots \ell_k^{i_k}}\\
&=\sum_{\mathbf{p}\in\mathcal{P}(k)}\zt\Big((\vec{t}^{p_1-1},\sum_{j_1=1}^{P_1}i_{j_1}), \dots, (\vec{t}^{p_r-1},\sum_{j_r=P_{r-1}+1}^{P_r}i_{j_r})\Big).
\end{align*}
\end{proof}

\section{Multi-interpolated MZVs with repeated arguments}
In the following we provide results for $\ztm(\{s\}_{k})$, generalizing earlier results for $\zt(\{s\}_{k})$~\cite{Bor}, as well as $\ztt(\{s\}_{k})$~\cite{HI}. 
Our proofs are based on generating functions and symbolic combinatorial constructions.

\begin{thm}
\label{TheGF}
The generating function $\Theta(z,\vec{t})=\sum_{k\ge 0}\ztm(\{s\}_{k})z^k$
of the multi-interpolated MZVs $\ztm(\{s\}_{k})$ is given by
\begin{equation*}
\Theta(z,\vec{t})=\prod_{m=1}^\infty\Big(1+\frac{\frac{1}{m^s}\cdot z}{1-\frac{1}{m^s} z t_m}\Big)
=\exp\Big(\sum_{j=1}^{\infty}\frac{z^j}{j}\cdot \big(\zt(\vec{t}^j, js)-\zt((\vec{t}-\vec{1})^j,js)\big)\Big).
\end{equation*}

\end{thm}

From the generating function we deduce several representations. 
We collect the definition of the complete Bell polynomials $B_n(x_1,\dots,x_n)$, determined by the identity
\[
\exp\Big(\sum_{\ell\ge 1}\frac{z^{\ell}}{\ell!}x_\ell\Big)
= \sum_{j\ge 0}\frac{B_j(x_1,\dots,x_j)}{j!}z^j,
\]
such that
\[
B_k(x_1,\dots,x_k)=\sum_{m_1+2m_2+\dots+k m_k=k}\frac{k!}{m_1!m_2!\cdot m_n!}\left(\frac{x_1}{1!}\right)^{m_1}\dots\left(\frac{x_k}{k!}\right)^{m_k}.
\]

\begin{coroll}
\label{Co1}
The values $\ztm(\{s\}_{k})$ satisfy
\[
\ztm(\{s\}_{k})=\sum_{\ell=0}^{k}\zts(\{(\vec{t},s)\}_{\ell})\cdot \zt\big(\{(\vec{1}-\vec{t},s)\}_{k-\ell}\big),
\]
where $\zt(\{(\vec{t},s)\}_{0})=\zts(\{(\vec{t},s)\}_{0})=1$. 

\smallskip

Moreover, 
\begin{equation*}
\ztm(\{s\}_{k})=\frac1{k!} B_k(x_1,\dots,x_k),
\end{equation*}
with $x_j=(j-1)!\Big(\zt((\vec{t}^j,j s))-\zt\big(((\vec{t}-\vec{1})^j,js)\big)\Big)$, $1\le j\le k$.
\end{coroll}

\begin{remark}
Both expressions are also true for the truncated variants. The Bell polynomial expression 
also leads to a third representation of $\ztm(\{s\}_{k})$ in terms of a determinant. A determinantal expression for $B_k(x_1,\dots,x_k)$ is given in~\cite{Collins2001} based on~\cite{Comtet,Ivanov}; another expression is obtained by using modified Bell polynomials 
$Q_k(x_1,\dots,x_k)$, given by
\begin{align*}
Q_k(x_1,\dots,x_k)&=\frac1{k!}B_k(0!x_1,1!x_2,\dots,(k-1)!x_k)\\
&=\sum_{m_1+2m_2+\dots+k m_k=k}\frac{1}{m_1!m_2!\cdot m_n!}\left(\frac{x_1}{1}\right)^{m_1}\dots\left(\frac{x_k}{k}\right)^{m_k}
\end{align*}
by MacDonald~\cite{MacDo} (see Hoffman~\cite{H2017} for additional properties): with
\[
Q_k(x_1,\dots,x_k)=\frac{1}{k!}\cdot
\left(
\begin{matrix}
x_1 & -1 & 0 &\dots &0\\
x_2 & x_1 & -2 &\dots &0\\
\vdots & \vdots & \vdots &\ddots &\vdots\\
x_{k-1} & x_{k-2} & x_{k-3} &\dots &-(k-1)\\
x_{k} & x_{k-1} & x_{k-2} &\dots &x_1\\
\end{matrix}
\right).
\]
\end{remark}

\begin{proof}[Proof of Theorem~\ref{TheGF}]
We use the symbolic constructions~\cite{FS2009}: let $\mathcal{Z}_m=\{m\}$ be a combinatorial class of size one, $1\le m\le n$. Due to the sequence construction we can describe the class of multisets $\mathcal{B}_m$ of $\mathcal{Z}_m$ as follows
\[
\mathcal{B}_m=\text{SEQ}(\mathcal{Z}_m)=\{\epsilon\} + \mathcal{Z}_m+ \mathcal{Z}_m\times \mathcal{Z}_m+\mathcal{Z}_m\times \mathcal{Z}_m\times \mathcal{Z}_m+\dots;
\]
Thus, the generating function 
\[
B_m(z)=\sum_{\beta\in\mathcal{B}_m}w(\beta)t^{\sigma(\beta)}z^{|\beta|}= 1+\sum_{\epsilon\neq \beta\in\mathcal{B}_m}w(\beta)t^{|\beta|-1}z^{|\beta|}
\]
with weight $w(\beta)=\frac{1}{m^{s\cdot |\beta|}}$, is given by
\[
B_m(z)=1+\sum_{j=1}^{\infty}t_m^{j-1}\frac{1}{(m^s)^j} z^{j}=1+\frac{\frac{1}{m^s} z}{1- \frac{1}{m^s}\cdot t_m z}.
\]
Let 
\[
\M_{n,k}=\{\vec{\ell}=(\ell_{k},\ell_{k-1}\dots,\ell_1)\in\N^k\colon 1\le \ell_k\le \dots \le \ell_2\le \ell_1 \le n\}.
\]
All multisets $\M_n=\bigcup_{k=1}^{\infty}\M_{n,k}$, 
with $k$-multisets of $\{1,2,\dots,n\}$ can be combinatorially generated by
\[
\M_n=\mathcal{B}_1\times \mathcal{B}_2\times\dots \times \mathcal{B}_n.
\]
Hence, the generating function $\Theta_n(z,\vec{t})$ is given by the stated formula. The result for the non-truncated multiple zeta values follow by taking the limit. Then, we use the $\exp-\log$ representation and the expansion of $\ln(1-z)$ to get
\begin{align*}
\Theta(z,\vec{t})&=\exp\left(\sum_{m=1}^\infty\ln\Big(1-\frac{z(t_m-1)}{m^s}\Big)-\ln\Big(1-\frac{z t_m}{m^s}\Big)\right)\\
&=\exp\left(\sum_{m=1}^\infty\sum_{j=1}^{\infty}\frac{z^j}{j}\cdot\frac{t_m^j-(t_m-1)^j}{m^{js}}\right)\\
&=\exp\left(\sum_{j=1}^{\infty}\frac{z^j}{j}\Big(\zt((\vec{t}^j,js))-\zt(((\vec{t}-\vec{1})^j,js))\Big)\right).
\end{align*}
\end{proof}

\begin{proof}[Proof of Corollary~\ref{Co1}]
From the expression for $\Theta(z,\vec{t})$ we get
\[
\ztm(\{s\}_{k})=[z^k]\Theta(z,\vec{t})=
[z^k]\left(\prod_{m=1}^\infty\Big(1+\frac{(1-t_m)z}{m^s}\Big)\right)
\left(\prod_{m=1}^\infty\frac1{1- \frac{z t_m}{m^s}}\right).
\]
The former expression is exactly the generating function 
of $\zt(\{(\vec{1}-\vec{t},s)\}_k)$, whereas the latter expression is the generating function 
of $\zts(\{(\vec{t},s)\}_k)$.
\end{proof}

\section{A product for multi-interpolated zeta values\label{SecAlg}}
We discuss algebraic properties of the multi-interpolated multiple zeta values. Following Hoffman~\cite{H2000,H2018,HI}
and Yamamoto~\cite{Y}, let $\mathcal{A}=\{z_1,z_2,\dots\}$ denote a countable set of letters.
Let $\mathbb{Q}\langle\mathcal{A}\rangle$ denote the rational non-commutative polynomial algebra 
and $\mathfrak{h}^{1}$ the underlying rational vector space of $\mathbb{Q}\langle\mathcal{A}\rangle$.
There are two products $\ast$ and $\star$ on $\mathfrak{h}^{1}$ defined by
\[
x \ast 1 =1 \ast x=x,\quad x \star 1 =1 \star x=x,
\]
For words $x=au$, $y=bv$ we have
\begin{equation}
\label{eqn:ast}
x\ast y= a(u\ast bv) + b(au\ast v) + a \diamond b (u\ast v),
\end{equation}
whereas
\begin{equation}
\label{eqn:star}
x\star y= a(u\star bv) + b(au\star v) - a \diamond b (u\star v).
\end{equation}
Here, $\diamond$ denotes the commutative product
\begin{equation}
\label{eqn:diamond}
z_i\diamond z_j=z_{i+j},
\end{equation}
and $z_i\diamond 1=1\diamond z_i=0$. The product $\ast$ corresponds to the multiplication of multiple zeta values,
whereas the $\star$ product to the multiple zeta star values.

\begin{remark}
It is well known that if the $\diamond$ product is defined trivially by $x\diamond y=0$, 
then both products reduce to the shuffle product: $\ast=\star=\shuffle$.
\end{remark}

Let $\mathfrak{h}^0$ be the subspace of $\mathfrak{h}^1$ generated by 1 and monomials that do not start with $z_1$, then the linear
map $Z \colon \mathfrak{h}^0\to\R$, defined by 
\begin{equation}
\label{MapZeta}
Z(z_{i_1}\dots z_{i_k})=\zt(i_1,\dots,i_k)
\end{equation}
and $Z(1)=1$ is a homomorphism from $(\mathfrak{h}^0; \ast)$ to the reals. 

\smallskip

Yamamoto~\cite{Y} introduced the interpolated product $\istuff$, generalizing~\eqref{eqn:ast} and~\eqref{eqn:star}:
$\stuff{1}=\star$, $\stuff{0}=\ast$. He also introduced a map $\Ztt$ from $\mathfrak{h}^{0}$ to $\R$. It maps a word $x=z_{i_1}\dots z_{i_k}\in \mathfrak{h}^0$ to an interpolated MZV:
\[
\Ztt(z_{i_1}\dots z_{i_k})=\ztt(i_1,\dots,i_k).
\]
Moreover, for $x,y\in \mathfrak{h}^0$ the product $\istuff$ satisfies the important relation
\[
\Ztt(x \istuff y)=\Ztt(x)\cdot \Ztt(y).
\]

\smallskip

The refinement of interpolated MZVs to multi-interpolated MZVs in Definition~\ref{DefMuli} 
suggests to look at corresponding generalizations of the map $\Ztt$ and the product $\istuff$. In the following we introduce a multi-interpolated product $\mstuff$, as well as a map $\Ztm$
such that for $x,y\in \mathfrak{h}^0$ it holds
\begin{equation}
\label{eqn:MainResult}
 \Ztm(x \mstuff y)=\Ztm(x)\cdot \Ztm(y).
\end{equation}

In order to make this precise, we first define the map $\Ztm$. For this purpose, we introduce a new variable, which we denote with $\vec{t}$. 
It does not commute with any letters $z_i$ of the alphabet $\mathcal{A}$. Let $\mathbb{Q}\langle\mathcal{A},\vec{t}\rangle$ denote the algebra of noncommutative polynomials
in the letters $z_1,z_2,\dots\in\mathcal{A}$ and the symbol $\vec{t}$, where we denote with  $\vec{t}^m$ a string consisting of $m$ occurrences of $\vec{t}$, $m\ge 0$. 
We denote with $\hneu$ the subspace, generated by 1 and monomials that do not start with
$\vec{t}^j z_1$, $j\ge 0$ or end with a letter $\vec{t}$. Non-empty words $x\in\hneu$ have the form
\begin{equation}
\label{word:x}
x=\vec{t}^{m_1}z_{i_1}\vec{t}^{m_2}z_{i_2}\dots \vec{t}^{m_n}z_{i_n},
\end{equation}
with $n\ge 1$ and $m_1,\dots m_1\ge 0$, where $i_1\neq 1$.

\smallskip

The map $\Ztm$ has domain $\hneu$ and codomain $\R[[\vec{t}]]$ and sends words $x\in \hneu$~\eqref{word:x} to multiple zeta values
with variables $t_1,t_2,\dots$, as introduced in Definition~\ref{DefNewZeta1}:
\begin{equation}
\label{def:Ztm}
\Ztm(\vec{t}^{m_1}z_{i_1}\vec{t}^{m_2}z_{i_2}\dots \vec{t}^{m_n}z_{i_n})
=\ztm\big((\vec{t}^{m_1},i_1),(\vec{t}^{m_2},i_2),\dots, (\vec{t}^{m_n},i_n)\big).
\end{equation}

\smallskip 

For the structural analysis of $\Ztm$ we will revisit the map $Z$~\eqref{MapZeta}, as well as the ordinary stuffle product $\ast$~\eqref{eqn:ast}.
We extend the domain and codomain of $Z$ to $\hneu$ and $\R[[\vec{t}]]$, respectively:
\begin{equation}
\label{MapZeta2}
Z(\vec{t}^{m_1}z_{i_1}\vec{t}^{m_2}z_{i_2}\dots \vec{t}^{m_n}z_{i_n})=\zt\big((\vec{t}^{m_1},i_1),(\vec{t}^{m_2},z_2)\dots (\vec{t}^{m_n},i_n)\big).
\end{equation}
\begin{example}
We emphasize the non-commutativity of the letter $\vec{t}$ with the letters $z_j\in\mathcal{A}$. Let $x=z_{i_1}\vec{t}z_{i_2}$ and  $y=\vec{t}z_{i_1}z_{i_2}$.
Then, 
\[
Z(x)=Z(z_{i_1}\vec{t}z_{i_2})=\sum_{\ell_1> \ell_2\ge 1}\frac{t_{\ell_2}} {\ell_1^{i_1}\ell_2^{i_2}}\neq
Z(y)=Z(\vec{t}z_{i_1}z_{i_2})=\sum_{\ell_1> \ell_2\ge 1}\frac{t_{\ell_1}} {\ell_1^{i_1}\ell_2^{i_2}}.
\]
In contrast for $\vec{t}=t$, see Remark~\ref{rem:sigma}, we would have obtained the same result $t\cdot\zeta(i_1,i_2)$, as 
\[
z_{i_1}t z_{i_2}=t z_{i_1}z_{i_2}.
\]
\end{example}

Next, define how ordinary stuffle product $\ast$~\eqref{eqn:ast} for MZVs acts on words in $\hneu$.
Let $x=\vec{t}^{m_1}a_1 u\in \hneu$ with $u=\vec{t}^{m_2}a_2\dots \vec{t}^{m_n}a_n$ and $y=\vec{t}^{k_1}b_1v\in \hneu$ with $v=\vec{t}^{k_2}a_2\dots \vec{t}^{k_r}b_r$,
with $a_\ell,b_j\in\mathcal{A}$, $1\le \ell\le n$ and $1\le j\le r$. 
Then, $x\ast y$ is defined by
\begin{equation}
\label{eqn:stuffleGeneral}
x\ast y=\vec{t}^{m_1}a_1(u\ast y) + \vec{t}^{k_1}b_1 (x\ast v) + (\vec{t}^{m_1}a_1 \diamond \vec{t}^{k_1}b_1) (u\ast v).
\end{equation}
where we extend the definition of the product $\diamond$~\eqref{eqn:diamond} to
\begin{equation}
\label{eqn:diamondNew}
\vec{t}^{c} z_i\diamond \vec{t}^{d} z_j= \vec{t}^{c+d}z_{i+j},\quad c,d\ge 0.
\end{equation}

\smallskip

Next, we introduce the product $\mstuff$. 
\begin{defi}[Product $\mstuff$]
Let $x,y\in\mathfrak{h}^0$ denote two words and $\vec{t}$ a variable non-commutative with the letters $z_k\in\mathcal{A}$. If $y=1$ then
\[
x \mstuff 1 =1 \mstuff x=x.
\]
For $x=a\in\mathcal{A}$ and $y=b\in\mathcal{A}$ single letter words we have
\[
x\mstuff y= a\mstuff b = ab + ba + (1-2\vec{t})a \diamond b.
\]
For words $x=au\in\mathfrak{h}^0$, $y=bv\in\mathfrak{h}^0$ we have
\[
x\mstuff y= a(u\mstuff bv) + b(au\mstuff v) + (1-2\vec{t})a \diamond b (u\mstuff v)
+(\vec{t}^2-\vec{t})a \diamond b \diamond (u\mstuff v).
\] 
\end{defi}
\begin{remark}
Our definition looks similar to the definition of Yamamoto~\cite{Y}.
Indeed, when the variable $\vec{t}$ is evaluated into $t$~\eqref{eqn:tgleicht}, then $\mstuff \ = \ \istuff$.
We emphasize again the key difference, namely the non-commutativity of $\vec{t}$ with the letters $z_k\in\mathcal{A}$ as .
\end{remark}

\begin{remark}
By~\eqref{def:ztm} and~\eqref{eqn:MainResult} we are mainly interested in words $x,y\in\mathfrak{h}^0$. 
However, we can readily extend the definition of $\mstuff$ to words $x,y\in\hneu$, compare with~\eqref{eqn:stuffleGeneral}:
\begin{equation*}
\begin{split}
x\mstuff y&= \vec{t}^{m_1}a_1(u\mstuff \vec{t}^{k_1}b_1 v) + \vec{t}^{k_1}b_1(\vec{t}^{m_1}a_1u\mstuff v) + (1-2\vec{t})\vec{t}^{m_1+k_1}a_1 \diamond b_1 (u\mstuff v)\\
&\quad+(\vec{t}^2-\vec{t})\vec{t}^{m_1+k_1} a_1 \diamond b_1 \diamond (u\mstuff v).
\end{split}
\end{equation*} 

\end{remark}

\begin{example}
\label{ExM1}
Let $x=a=z_i$ and $y=b=z_j$. 
Then 
\[
x\mstuff y= z_i\mstuff z_j = z_iz_j + z_jz_i + (1-2\vec{t})z_{i+j}.
\]
\end{example}

\begin{example}
\label{ExM2}
Let $x=a=z_i$ and $y=bz_k=z_jz_k$. 
\begin{align*}
x\mstuff y&= z_iz_jz_k + z_j(z_i\mstuff z_k) + (1-2\vec{t})z_{i+j}z_k
+(\vec{t}^2-\vec{t})z_{i+j+k}\\
&=z_iz_jz_k+ z_j z_iz_k + z_jz_kz_i + z_j(1-2\vec{t})z_{i+k}\\
&\quad+ (1-2\vec{t})z_{i+j}z_k+(\vec{t}^2-\vec{t})z_{i+j+k}.
\end{align*}
\end{example}

\smallskip

In order to obtain the result~\eqref{eqn:MainResult} we need more insight into the map $\Ztm$.
We will realize the map $\Ztm$ using the map $Z$~\eqref{MapZeta},~\eqref{MapZeta2} and a new operator $S^{\vec{t}}$, generalizing the interpolating operator in~\cite{Y}.

\begin{defi}[Multi-interpolation operator]
For the empty word 1 and a single letter $a\in\mathcal{A}$ the multi-interpolation operator $S^{\vec{t}}$ is defined as
\[
S^{\vec{t}} (1)=1,\quad S^{\vec{t}}(a)=a.
\]
For $x=au\in\mathfrak{h}^0$, where $u$ denotes a substring, we set
\begin{equation}
\label{def:multIntOp}
S^{\vec{t}}(x)=S^{\vec{t}}(au)=a S^{\vec{t}}(u)+\vec{t} a\diamond S^{\vec{t}}(u).
\end{equation}
\end{defi}
\begin{remark}
Of course, one can also consider words $x\in\hneu$,
\[
x=\vec{t}^{m_1}a_1 u,\quad u=\vec{t}^{m_2}a_2\dots \vec{t}^{m_n}a_n,
\]
where we define
\begin{equation}
\label{def:multIntOpGen}
S^{\vec{t}}(x)=S^{\vec{t}}(\vec{t}^{m_1}a_1u)=\vec{t}^{m_1}a_1 S^{\vec{t}}(u)+\vec{t}^{m_1+1} a_1\diamond S^{\vec{t}}(u).
\end{equation}
\end{remark}

\begin{example}
Let $x=z_i z_j z_k$ denote word of length three. 
Then
\begin{align*}
S^{\vec{t}}(x)&=z_i S^{\vec{t}}(z_j z_k) + \vec{t}z_i\diamond S^{\vec{t}}(z_j z_k) \\
&=z_i z_j z_k + z_i \vec{t} z_j\diamond z_k + \vec{t}z_i\diamond z_j z_k + \vec{t}z_i\diamond \vec{t} z_j\diamond z_k\\
&=z_i z_j z_k + z_i \vec{t} z_{j+k} + \vec{t}z_{i+j} z_k + \vec{t}^2z_{i+j+k}.
\end{align*}
This corresponds exactly to the decomposition of the multi-interpolated multiple zeta value $\ztm(i,j,k)$ in Example~\ref{ExDeco3} 
and we see that
\[
(Z\circ S^{\vec{t}})(z_i z_j z_k)=\Ztm(z_i z_j z_k).
\]
\end{example}

\smallskip

In order to describe $S^{\vec{t}}(x)$ for a word $x$ let $R_n$, $n\in\N$, denote the set of subsequences $r = (r_0, \dots, r_s)$
of $(0,\dots, n)$ such that $r_0 = 0$ and $r_s = n$. For such $r$ and a word
$x = a_1\dots a_n$, we define the word $\Cont{r}(x)$ with respect to $\vec{t}$. 
It is the weighted contraction of $x$ with respect to $r$, weighted according 
\begin{equation}
\label{def:contr}
\Cont{r}(x) = \vec{t}^{r_{1}-r_0-1}b_1\cdot \vec{t}^{r_{2}-r_1-1}b_2 \cdots \vec{t}^{r_{s}-r_{s-1}-1}b_s, \quad b_i = a_{r_i+1}\diamond \dots \diamond a_{r_{i+1}} .
\end{equation}

\begin{lem}[Properties - Multi-interpolation operator]
\label{lem:InterpolationOperator}
Let $x = a_1\dots a_n$. The operator $S^{\vec{t}}$ has the properties
\begin{enumerate}[label=\alph*)]
\item $S^{\vec{t}}(x)=\sum_{r\in R_n}\Cont{r}(x)$,
	\item $(S^{\vec{t}}-1)^n x=0$,
	\item Assume that symbols $\vec{t}_1$, $\vec{t}_2$ commute with each other, but are non-commutative with respect to letters
	$z_k\in\mathcal{A}$. Let $\circ$ denote the operator composition, then
	\[
		S^{\vec{t}_1+\vec{t}_2}=S^{\vec{t}_1} \circ S^{\vec{t}_2}.
	\]
	\item Let $b\in\mathcal{A}$:
	\[
	b\diamond S^{\vec{t}}(x)=S^{\vec{t}}(b \diamond x).
	\]
	\end{enumerate}
\end{lem}

\begin{remark}
Note that for $\vec{t}=t$ the weighted contraction simplifies to
\[
\Cont{r}(x)=t^{\sum_{i=1}^{s}(r_{i}-r_{i-1}-1)}\Con{r}(x)=t^{n-s}\Con{r}(x)=t^{\sigma(r)}\Con{r}(x),
\]
where $\Con{r}(x)=b_1\cdots b_s$ denotes the standard contraction operator of~\cite{Y}.
\end{remark}

\begin{proof}[Proof of Lemma~\ref{lem:InterpolationOperator}]
The first part follows directly from the definition. 
Part (b) is shown using induction. We actually prove the stronger statement that
$(S^{\vec{t}}-1)^n x=0$ for $x=\vec{t}^{m_1}a_1\vec{t}^{m_2}a_2\dots \vec{t}^{m_n}a_n\in\hneu$.
 The statement is obviously true for an empty word $1$, as well as a single letter word $x=\vec{t}^{m}a$. 
Let $x=\vec{t}^{m_1}a_1u$ with $a_1$ a single letter and $u=\vec{t}^{m_2}a_2\dots \vec{t}^{m_n}a_n$ the subword.
By 
\[
(S^{\vec{t}}-1)^n x=\vec{t}^{m_1}(S^{\vec{t}}-1)^{n-1}(S^{\vec{t}}-1)a_1u.
\]
By a slight generalization of part a we observe $(S^{\vec{t}}-1)a_1u$ is a combination of words of length less or equal $n-1$.
Hence, by the induction hypothesis the expression is equal to zero. For part (c) we have to argue in combinatorial way, reducing the identity to a counting problem. 
By definition part a) and~\eqref{def:contr}, we have 
\[
S^{\vec{t}_1+\vec{t}_2}(x)=\sum_{r\in R_n}(\vec{t}_1+\vec{t}_2)^{r_{1}-r_0-1}b_1 \cdots (\vec{t}_1+\vec{t}_2)^{r_{s}-r_{s-1}-1}b_s.
\]
Assume that $1\le i\le s$ and $\rho_i=r_{i}-r_{i-1}-1$.
Expanding by the binomial theorem gives 
\[
(\vec{t}_1+\vec{t}_2)^{\rho_i}=\sum_{\ell=0}^{\rho_i} \binom{\rho_i}{\ell} \vec{t}_1^{\rho_i-\ell} \vec{t}_2^{\ell}.
\]
For each resulting block $b_i= a_{r_i+1}\diamond \dots \diamond a_{r_{i+1}}$, there have to be exactly $\rho_i$ contractions leading to it.
Out of these $\rho_i$ contractions, we can choose $0\le \ell\le \rho_i$ of them to stem from the application of $S^{\vec{t}_2}$. 
The subblocks leading to $b_i$ all have the common prefactor $\vec{t}_2^{\ell}$, 
as already $\ell$ contractions have occurred. As there are in total $\rho_i$ letters merged into $b_i$, we have $\binom{\rho_i}{\ell}$ words in $\Cont{r}(x)$, with $r\in R_n$, 
leading after application of $S^{\vec{t}_1}$ to the block $b_i$. The application of $S^{\vec{t}_1}$ to these words leads to an additional factor $\vec{t}_1^{\rho_i-\ell}$, 
as we contract the remaining letters to obtain $b_i$. 
Concerning the final statement (d) we observe that
\begin{align*}
	S^{\vec{t}}(b\diamond x)&
	= S^{\vec{t}}(b\diamond \vec{t}^{m_1}a_1u)
	= S^{\vec{t}}\big(\vec{t}^{m_1} b\diamond a_1 u\big )\\
	&= \vec{t}^{m_1} b\diamond a_1 S^{\vec{t}}(u)+\vec{t}^{m_1+1} b\diamond a_1 \diamond S^{\vec{t}}(u)\\
	&=b\diamond \big(\vec{t}^{m_1} a_1 S^{\vec{t}}(u)+\vec{t}^{m_1+1} a_1 \diamond S^{\vec{t}}(u)\big)
	=b\diamond S^{\vec{t}}(x). 
\end{align*}
\end{proof}

We also collect another property of $\mathcal{S}^{\vec{t}}$ when combined a suitably defined derivative. 
\begin{defi}[Derivative for symbol $\vec{t}$]
The action of the differential operator $\frac{d}{d\vec{t}}$ to $x\in\hneu$ is given as follows. For $x=a\in\mathfrak{h}^0$ we have 
\[
\frac{d}{d\vec{t}} a=0,\quad \frac{d}{d\vec{t}} \vec{t}^m a=m\cdot \vec{t}^{m-1} a.
\]
For $x=\vec{t}^{m_1}a_1 u\in \hneu$ with $u=\vec{t}^{m_2}a_2\dots \vec{t}^{m_n}a_n$ with $n\ge 2$ we define
\[
\frac{d}{d\vec{t}} x =m_1\vec{t}^{m_1-1}a_1 u + \vec{t}^{m_1}a_1 \frac{d}{d\vec{t}} u,
\]
\end{defi}
We observe the following. 
\begin{lem}
\label{lem:derivative}
Let $x=a_1\dots a_n$. The differential operator $\frac{d}{d\vec{t}}$ acting on $S^{\vec{t}}(x)$ can be decomposed into contractions:
\[
\frac{d}{d\vec{t}}S^{\vec{t}}(x)=\sum_{k=1}^{n-1}S^{\vec{t}}(a_1\dots a_{k-1} a_{k}\diamond a_{k+1} a_{k+2}\dots a_n).
\]
\end{lem}

\begin{proof}
We provide two different proofs. First, we argue in a combinatorial way. By Lemma~\ref{lem:InterpolationOperator} (a) each word in $S^{\vec{t}}(x)$
has the form 
\[
\Cont{r}(x) = \vec{t}^{r_{1}-r_0-1}b_1\cdot \vec{t}^{r_{2}-r_1-1}b_2 \cdots \vec{t}^{r_{s}-r_{s-1}-1}b_s, \quad b_i = a_{r_i+1}\diamond \dots \diamond a_{r_{i+1}} .
\]
Application of the derivative to $\Cont{r}(x)$ leads to $s$ words, each with multiplicative factor $\rho_i=r_i-r_{i-1}-1$ and the corresponding power of $\vec{t}$ at position $i$ diminished by one. 
Out of the words in $\sum_{k=1}^{n-1}S^{\vec{t}}(a_1\dots a_{k-1}a_{k}\diamond a_{k+1}\dots a_n)$
exactly the words with $r_{i-1}\le k\le r_{i}$ can be mapped to the resulting word in $\frac{d}{d\vec{t}}\Cont{r}(x)$.
Since every word in $\frac{d}{d\vec{t}}\Cont{r}(x)$ is covered like this, we obtain the stated result.

\smallskip 

Our second proof uses induction. We use the shorthand notation $a_{i,j}=a_i\dots a_j$ for $i\le j$.
We have
\[
\frac{d}{d\vec{t}}S^{\vec{t}}(x)
=\frac{d}{d\vec{t}}\Big(a_1 S^{\vec{t}}(a_{2,n})+\vec{t}a_1 \diamond S^{\vec{t}}(a_{2,n})\Big)
=a_1\frac{d}{d\vec{t}}S^{\vec{t}}(a_{2,n})+\frac{d}{d\vec{t}}\vec{t}a_1\diamond S^{\vec{t}}(a_{2,n}).
\]
The term $S^{\vec{t}}(a_{2,n})$ is directly covered by the induction hypothesis. By the ordinary product rule and by Lemma~\ref{lem:InterpolationOperator} (d) we get further
\[
\frac{d}{d\vec{t}}\vec{t}a_1\diamond S^{\vec{t}}(a_{2,n})
=S^{\vec{t}}(a_1\diamond a_{2}a_{3,n})+\vec{t}\frac{d}{d\vec{t}}S^{\vec{t}}\big((a_1\diamond a_{2})a_{3,n}\big).
\] 
The induction hypothesis applies to the second summand and we get
\begin{align*}
&\frac{d}{d\vec{t}}S^{\vec{t}}(a_1\diamond a_{2,n})
= \frac{d}{d\vec{t}}S^{\vec{t}}\big((a_1\diamond a_2) a_{3,n}\big)\\
&= S^{\vec{t}}\big((a_1\diamond a_2 \diamond a_{3})a_{4,n}\big)+\dots + S^{\vec{t}}\big((a_1\diamond a_2)a_{3,n-2}(a_{n-1}\diamond a_{n})\big).
\end{align*}
Collecting all contributions, the induction hypothesis implies that
\begin{align*}
\frac{d}{d\vec{t}}S^{\vec{t}}(x)&
=a_1\sum_{k=2}^{n}S^{\vec{t}}\big(a_{2,k-1} (a_k\diamond a_{k+1})a{k+2,n}\big) +S^{\vec{t}}\big((a_1\diamond a_{2})a_{3,n}\big)\\
&+\vec{t}\Big( S^{\vec{t}}\big((a_1\diamond a_2 \diamond a_{3})a_{4,n}\big)+\dots + S^{\vec{t}}\big((a_1\diamond a_2)a_{3,n-2}(a_{n-1}\diamond a_{n})\big)\Big).
\end{align*}
On the other hand, we can decompose the righthandside of the equation.
\begin{align*}
&\sum_{k=1}^{n-1}S^{\vec{t}}\big(a_{1,k-1}(a_{k}\diamond a_{k+1})a_{k+2,n}\big)
=S^{\vec{t}}((a_{1}\diamond a_2) a_{3,n})\\
&\quad+\sum_{k=2}^{n-1}S^{\vec{t}}\big(a_{1,k-1}(a_{k}\diamond a_{k+1})a_{k+2,n}\big)
\end{align*}
Furthermore,
\begin{align*}
S^{\vec{t}}\big(a_{1,k-1}a_{k}\diamond a_{k+1,n}\big)&=
a_1 S^{\vec{t}}\big(a_{2,k-1}(a_{k}\diamond a_{k+1})a_{k+2,n}\big)\\
&\quad+\vec{t} S^{\vec{t}}\big((a_1\diamond a_{2})a_{3,k-1}(a_{k}\diamond a_{k+1})a_{k+2,n}\big).
\end{align*}
This provides the induction step and thus the stated result.
\end{proof}


A direct application of Lemma~\ref{lem:derivative} is a generalization of the so-called alternating sum formula~\cite{IKOO}:
\begin{equation}
\label{eqn:alternatingOriginal}
\sum_{k=0}^{n}(-1)^k (a_1\cdots a_k) \ast S(a_n\cdots a_{k+1})=0.
\end{equation}
This formula was extended to interpolated MZVs~\cite[Proposition 3.7]{Y}
\[
\sum_{k=0}^{n}(-1)^k S^{t}(a_1\cdots a_k) \ast S^{1-t}(a_n\cdots a_{k+1})=0. 
\]
\begin{thm}[Alternating sum formula]
The multi-interpolation operator $S^{\vec{t}}$ satisfies
\[
\sum_{k=0}^{n}(-1)^k S^{\vec{t}}(a_1\cdots a_k) \ast S^{\vec{1}-\vec{t}}(a_n\cdots a_{k+1})=0. 
\]
\end{thm}
\begin{proof}
We use induction with respect to $n$. For $n=1$ we have
\[
S^{\vec{t}}(a_1)\cdot 1 - S^{\vec{1}-\vec{t}}(a_1)\cdot 1=a_1-a_1=0.
\]
For $n=2$ observe that
\begin{align*}
&S^{\vec{t}}(a_1a_2)\cdot 1 - S^{\vec{t}}(a_1)S^{\vec{1}-\vec{t}}(a_2)+ S^{\vec{1}-\vec{t}}(a_2a_1)\\
&\quad =a_1a_2 +\vec{t}a_1\diamond a_2 - \big(a_1a_2 +a_2a_1 -a_1\diamond a_2 \big) + a_2a_1+(\vec{1}-\vec{t})a_2\diamond a_1=0.
\end{align*}
Assume now that $n\ge 2$. We observe that the left-hand side is given by a sum of words of the form $x=\vec{t}^m_1b_1\dots \vec{t}^m_kb_k$, with $1\le k\le n$ 
plus words without any powers of $\vec{t}$. By the original alternating sum formula~\eqref{eqn:alternatingOriginal} we know that 
for $\vec{t}=0$ the truth of the statement. 
If there are words of a specific length $k$ of the form $x=\vec{t}^m_1b_1\dots \vec{t}^m_kb_k$,
such that their sum does not vanish, then the derivative $\frac{d}{d\vec{t}}$ 
results in words of the form $\vec{t}^m_1b_1\dots \vec{t}^{m_{r}-1} b_r\dots \vec{t}^m_kb_k$, $1\le r\le k$, 
which cannot vanish either. Hence, it suffices to prove that 
\[
\frac{d}{d\vec{t}} \sum_{k=0}^{n}(-1)^k S^{\vec{t}}(a_{1,k}) \ast S^{\vec{1}-\vec{t}}(A_{n,k+1})=0,
\]
where used again the shorthand notation $a_{i,j}=a_i\dots a_j$ for $i\le j$ and also $A_{j,i}=a_j\dots a_i$.
We can now proceed identically to~\cite{Y} and obtain for this derivative of the left-hand side by Lemma~\ref{lem:derivative} and the Leibniz rule
\begin{align*}
&\sum_{k=2}^{n}(-1)^k \sum_{i=1}^{k-1} S^{\vec{t}}(a_{1,i-1}a_i\diamond a_{i+1} a_{i+2,k}) \ast S^{\vec{1}-\vec{t}}(A_{n,k+1})\\
&\quad - \sum_{k=0}^{n-2}(-1)^k \sum_{i=k+1}^{n-1} S^{\vec{t}}(a_{1,k}) \ast S^{\vec{1}-\vec{t}}(A_{n,i+2}a_{i+1}\diamond a_i A_{i-1,k} )\\
&=-\sum_{i=1}^{n-1}\bigg(\sum_{k=0}^{i-1}(-1)^k S^{\vec{t}}(a_{1,k})\ast S^{\vec{1}-\vec{t}}(A_{n,i+2}a_{i+1}\diamond a_i A_{i-1,k} )\\
&\quad + \sum_{k=i+1}^{n}(-1)^{k-1} S^{\vec{t}}(a_{1,i-1}a_i\diamond a_{i+1} a_{i+2,k}) \ast S^{\vec{1}-\vec{t}}(A_{n,k+1}).
\bigg)
\end{align*}
By the induction hypothesis the sums inside the brackets vanish, which proves the stated result.
\end{proof}

Now we turn to the statement of our main result.

\begin{thm}
Let $\mstuff$ denote the commutative product on the algebra $\mathbb{Q}\langle\mathcal{A}\rangle[\vec{t}]$. 
The map $S^{\vec{t}}$ is an isomorphism from $(\hneu,\mstuff)$ to $(\hneu,\mstuff)$, 
such that
\[
S^{\vec{t}}(x\mstuff y)=S^{\vec{t}}(x)\ast S^{\vec{t}}(y).
\]

\smallskip

The map $\Ztm$~\eqref{def:Ztm} can be decomposed into the map $Z$~\eqref{MapZeta},~\eqref{MapZeta2} 
and the multi-interpolation operator $S^{\vec{t}}$~\eqref{def:multIntOp}, such that $\Ztm =Z \circ S^{\vec{t}}$.
Moreover,
\[
 \Ztm(x \mstuff y)= \Ztm(x)\cdot  \Ztm(y).
\]
\end{thm}

\begin{remark}
If we define $\ztm\big( (i_1,\dots,i_k)\istuff (j_1,\dots,j_\ell)\big)$ as 
$ \Ztm(x \mstuff y)$ with $x=z_{i_1}\dots z_{i_k}$ and $y=z_{j_1}\dots z_{j_\ell}$, 
then we write 
\[
\ztm\big( (i_1,\dots,i_k)\istuff (j_1,\dots,j_\ell)\big)=\ztm(i_1,\dots,i_k)\cdot\ztm( j_1,\dots,j_\ell),
\]
compare with~\eqref{eqn:YamamotoIdentity}.
\end{remark}

\begin{example}
Let $x=a=z_i$ and $y=b=z_j$. By Example~\ref{ExM1} we get
\begin{align*}
\Ztm(x \mstuff y)&= \Ztm(z_iz_j) + \Ztm(z_jz_i) + \Ztm\big((1-2\vec{t})z_{i+j}\big)\\
&=\ztm(i,j) + \ztm(j,i) + \sum_{\ell \ge  1}\frac{1-2t_{\ell}} {\ell^{i+j}}\\
&=\zt(i,j) +\zt\big((i+j)\vt\big) +\zt(j,i) +\zt\big((i+j)\vt\big) + \sum_{\ell \ge  1}\frac{1-2t_{\ell}} {\ell^{i+j}}\\
&=\zt(i,j)+\zt(j,i)+\zt(i+j)=\ztm(i)\ztm(j)=\Ztm(z_i)\cdot \Ztm(z_j),
\end{align*}
since $\ztm(i)=\zt(i)$ and $\ztm(j)=\zt(j)$.
\end{example}

\begin{example}
Let $x=a=z_i$ and $y=bz_k=z_jz_k$. By Example~\ref{ExM2} we get
\begin{align*}
\Ztm(x\mstuff y) &=\zt(i,j,k)+\zt(i+j,k)+\zt(j,i,k)+\zt(j,i+k)+\zt(j,i,k)\\
&\quad + \zt(i,\vec{t}(j+k))+ \zt(\vec{t}(i+j+k))+ \zt(\vec{t}(j+k),i),
\end{align*}
which is exactly the product of $\Ztm(z_i)=\zt(i)$ and $\Ztm(z_jz_k)=\zt(j,k)+\zt((\vec{t},j+k))$.
\end{example}

\begin{proof}
By induction with respect to the number of letters $z_k\in\mathcal{A}$, te multi-interpolation map $S^{\vec{t}}$ is injective. 
It is also surjective, as $S^{\vec{t}}\big(S^{-\vec{t}}(x)\big)=S^{0}(x)=x$.
Once, the homomorphism is established, the result for $\Ztm$ follows directly by Theorem~\ref{the:decomposition} and Lemma~\ref{lem:InterpolationOperator},
as well as an application of $Z$. In order to prove the homomorphism,  
we follow closely the strategy of~\cite{Y} and also use results of, so one should give much credit to these works.
For the sake of simplicity, we present the proof solely for words $x,y\in\mathfrak{h}^0$, as the general case $x,y\in\hneu$ 
is more involved, to the additional occurrences~\eqref{word:x} of $\vec{t}$.
We proceed by induction with respect to the total length of the words. 
For $x=1$ or $y=1$ or both, this is obviously true. 
For $x=a$ and $y=b$, corresponding to Example~\ref{ExM1},  we have
\begin{align*}
S^{\vec{t}}(a\mstuff b)&=S^{\vec{t}}\big(ab + b a + (1-2\vec{t})a \diamond b\big) \\
&=ab + \vec{t}a\diamond b + b a +\vec{t} b\diamond a + (1-2\vec{t})a \diamond b\\
&=ab+b a +2\vec{t} a\diamond b + a\diamond b -2\vec{t}a \diamond b\\
&=ab+b a + a\diamond b =S^{\vec{t}}(a)\ast S^{\vec{t}}(b).
\end{align*}

Assume that $x=au$ and $y=bv$.
On one hand, we have
\[
S^{\vec{t}}(x)\ast S^{\vec{t}}(y)
=\big(aS^{\vec{t}}(u) + \vec{t}a\diamond S^{\vec{t}}(u)  \big)\ast \big(bS^{\vec{t}}(v) + \vec{t}b\diamond S^{\vec{t}}(v) \big).
\]
Let $U=S^{\vec{t}}(u)$ and $V=S^{\vec{t}}(v)$.
Thus,
\begin{align*}
S^{\vec{t}}(x)\ast S^{\vec{t}}(y)
=\big(aU + \vec{t}a\diamond U  \big)\ast \big(bV + \vec{t}b\diamond V \big).
\end{align*}
Expanding the term above gives
\begin{align}
\label{EqEqual}
S^{\vec{t}}(x)\ast S^{\vec{t}}(y)&=(aU) \ast (bV) + (aU)\ast (\vec{t}b\diamond V)\\ \nonumber
&\quad +(\vec{t}a\diamond U) \ast (bV)+(\vec{t}a\diamond U) \ast (\vec{t}b\diamond V).
\end{align}

On the other hand, by definition of $\mstuff$ we have
\begin{align*}
S^{\vec{t}}(x\mstuff y)&=S^{\vec{t}}\big(a(u\mstuff b v)\big) + S^{\vec{t}}\big(b(a u\mstuff v)\big) + S^{\vec{t}}\big((1-2\vec{t})a \diamond b (u\mstuff v)\big)\\
&\quad+S^{\vec{t}}\big((\vec{t}^2-\vec{t})a \diamond b \diamond (u\mstuff v)\big).
\end{align*}
By the definition of $S^{\vec{t}}$ we get further
\begin{align*}
S^{\vec{t}}(x\mstuff y)&=a S^{\vec{t}}\big(u\mstuff b v\big) + \vec{t}a\diamond S^{\vec{t}}\big(u\mstuff b v\big) \\
&\quad+ b S^{\vec{t}}\big(a u\mstuff v\big) + \vec{t}b\diamond S^{\vec{t}}\big(a u\mstuff v\big)\\
&\quad+ (1-2\vec{t})a \diamond b S^{\vec{t}}\big(u\mstuff v\big)+(1-2\vec{t})\vec{t}a \diamond b \diamond S^{\vec{t}}\big(u\mstuff v\big)\\
&\quad+S^{\vec{t}}\big((\vec{t}^2-\vec{t})a \diamond b \diamond (u\mstuff v)\big).
\end{align*}
Simplification by Lemma~\ref{lem:InterpolationOperator} gives
\begin{align*}
S^{\vec{t}}(x\mstuff y)&=a S^{\vec{t}}\big(u\mstuff b v\big) + \vec{t}a\diamond S^{\vec{t}}\big(u\mstuff b v\big) \\
&\quad+ b S^{\vec{t}}\big(a u\mstuff v\big) + \vec{t}b\diamond S^{\vec{t}}\big(a u\mstuff v\big)\\
&\quad+ (1-2\vec{t})a \diamond b S^{\vec{t}}\big(u\mstuff v\big)-\vec{t}^2 a\diamond b \diamond S^{\vec{t}}(u\mstuff v)\big).
\end{align*}
Let $U=S^{\vec{t}}(u)$ and $V=S^{\vec{t}}(v)$. By the induction hypothesis, we get further
\begin{align*}
S^{\vec{t}}(x\mstuff y)&=a \big(U\ast S^{\vec{t}}(b v)\big) + \vec{t}a\diamond \big(U\ast S^{\vec{t}}(b v)\big)\\
&\quad+ b S^{\vec{t}}(au)\ast S^{\vec{t}}(v) + \vec{t}b\diamond S^{\vec{t}}\big(a u\mstuff v\big)\\
&\quad+ (1-2\vec{t})(a \diamond b) (U\ast V) - \vec{t}^2 a\diamond b \diamond S^{\vec{t}}(U\ast V).
\end{align*}
Using again the definition of $S^{\vec{t}}$, we obtain
\begin{align*}
S^{\vec{t}}(x\mstuff y)&=a \big(U\ast (bV)\big) + a \big(U\ast (\vec{t} b\diamond V)\big) \\
&\quad + \vec{t}a\diamond (U\ast bV)+ \vec{t}a\diamond \big(U\ast (\vec{t} b\diamond V)\big)\\
&\quad+ b \big((a U)\ast V\big) + b \big((\vec{t} a\diamond U)\ast V\big) \\
&\quad+ \vec{t}b\diamond \big((a U)\ast V\big) +  \vec{t}b\diamond \big((\vec{t}a \diamond U)\ast V\big)\\
&\quad+ (1-2\vec{t})(a \diamond b) (U\ast V) - \vec{t}^2(a\diamond b \diamond (U\ast V)\big).
\end{align*}
We use the following identities of~\cite{IKOO} (see also \cite{Y}):
\begin{align*}
(a\diamond U)\ast (bV)&= a\diamond\big(U\ast (bV)\big)+b\big((a\diamond U)\ast V\big) - (a\diamond b)(U\ast V),\\
(aU)\ast (b\diamond V)&= b\diamond \big((a U)\ast V\big)+a\big(U\ast (b\diamond V)\big)- (a\diamond b)(U\ast V),\\
(a\diamond U)\ast (b\diamond V)&=a\diamond \big(U\ast(b\diamond V)\big)+b\diamond\big((a\diamond U)\ast V\big)- (a\diamond b)\diamond(U\ast V).
\end{align*}
They imply that
\begin{align*}
&\vec{t}a\diamond\big(U\ast (bV)\big)+b\big((\vec{t}a\diamond U)\ast V\big)= (\vec{t}a\diamond U)\ast (bV) + \vec{t}(a\diamond b)(U\ast V),\\
&\vec{t}b\diamond \big((a U)\ast V\big)+a\big(U\ast (\vec{t}b\diamond V)\big)= (aU)\ast (\vec{t}b\diamond V)+ \vec{t}(a\diamond b)(U\ast V),\\
&\vec{t}a\diamond \big(U\ast(\vec{t}b\diamond V)\big)+\vec{t}b\diamond\big((\vec{t}a\diamond U)\ast V\big)\\
&\qquad=(\vec{t}a\diamond U)\ast (\vec{t}b\diamond V)+ \vec{t}^2(a\diamond b)\diamond(U\ast V).
\end{align*}
Together with the definition
\[
(aU)\ast (bV)=a\big(U\ast (bV)\big) + b((aU)\ast V)+ (a\diamond b)(U\ast V),
\]
we get the desired result
\begin{align*}
S^{\vec{t}}(x\mstuff y)&= (\vec{t}a\diamond U)\ast (bV)  + (aU)\ast (\vec{t}b\diamond V)\\
&\qquad +(\vec{t}a\diamond U)\ast (\vec{t}b\diamond V)+(aU)\ast (bV),
\end{align*}
which equals $S^{\vec{t}}(x)\ast S^{\vec{t}}(y)$~\eqref{EqEqual}.
In the general case, instead of $x=au$ and $y=bv$ we have $x=\vec{t}^{m}a u\in \hneu$ with $u=\vec{t}^{m_2}a_2\dots \vec{t}^{m_n}a_n$ and $y=\vec{t}^{k}bv\in \hneu$ with $v=\vec{t}^{k_2}a_2\dots \vec{t}^{k_r}b_r$,
with $a,b,a_\ell,b_j\in\mathcal{A}$, $2\le \ell\le n$ and $2\le j\le r$. We can proceed very similar so the arguments are omitted.
\end{proof}

\section{Outlook and summary}
\subsection{Open problems: properties of multi-interpolation}
At the moment, it seems difficult to translate more properties for $\ztt(\vec{i})$ to $\ztm(\vec{i})$. 
Yamamoto established, amongst many other things, the sum property for interpolated multiple zeta values:
\begin{equation}
\label{YSum}
\sum_{\substack{k_1\ge 2, k_i\ge 1\\ \sum_{\ell=1}^{n}k_\ell =k}}\ztt(k_1,\dots,k_n)
=\zt(k)\cdot \sum_{j=0}^{n-1}\binom{k-1}j t^j(1-t)^{n-1-j},
\end{equation}
with $k>n$. For $\ztm(\vec{i})$ we expect that only the simplest case $n=2$ can be easily treated in full generality:
\[
\sum_{k_1=2}^{k-1}\ztm(k_1,k-k_1)=\zt(k) + (k-2)\zt(k\vt).
\]
It is certainly desirable to obtain purely algebraic proofs of Theorem~\ref{TheGF} and Corollary~\ref{Co1}. Moreover, as the new definitions combine both interpolated MZVs
as well as multiple $\mathfrak{t}$-values, it is of interest to look deeper into a common framework for these objects.

\smallskip

\subsection{Outlook: multi-interpolation based on ordered partitions}
One can also look at different kinds interpolations. Extraction of the coefficient of $t^i$ in Yamamoto's sum formula implies Yamamoto's refined identity (see also~\cite{Li,LiQin1}):
\begin{equation}
\label{YamIdentity}
\sum_{\substack{k_1\ge 2, k_i\ge 1\\ \sum_{\ell=1}^{n}k_\ell =k}}
\sum_{\substack{\mathbf{p}_n\in\mathcal{P}(n)\\ n-\mathcal{L}(\mathbf{p}_n)=i}}
\zt\big(\mathbf{p}(k_1,\dots,k_n)\big)
=\zt(k)\cdot \binom{k-n+i-1}{k-n-1}.
\end{equation}
For example, 
\[
\sum_{\substack{k_1\ge 2, k_i\ge 1\\ \sum_{\ell=1}^{n}k_\ell =k}}
\Big(\zt\big(k_1+k_2,\dots,k_n)+\dots+\zt\big(k_1,\dots,k_{n-1}+k_n)\big)=(k-n)\zt(k).
\]

This motivates a different definition of multi-interpolated MZVs, based on ordered partitions:
\begin{equation}
\label{InterpolationU}
\zt^{(\vec{u})}(k_1,\dots,k_n)=
\sum_{\mathbf{p}_n\in\mathcal{P}(n)}
\bigg(\prod_{m=1}^{\infty}u_{m-1}^{N_m(\mathbf{p}_n)}\bigg)
\zt(\mathbf{p}_n(\vec{k})),
\end{equation}
where $N_m(.)$ counts the number of parts of size $m$ in an ordered partition. Here we set $u_0=1$. Since 
\[
\sum_{m=1}^{\infty}(m-1)\cdot N_m(\mathbf{p}_n)=n-\mathcal{L}(\mathbf{p}_n)
\]
we reobtain for $u_m=t^{m}$ the ordinary interpolated MZVs. For example, 
\[
\zt^{(\vec{u})}(k_1,k_2,k_3)=
u_0^3\zt(k_1,k_2,k_3)+u_0u_1\big(\zt(k_1+k_2,k_3)+\zt(k_1,k_2+k_3)\big)
+u_2\zt(k_1+k_2+k_3).
\]
Note that this new definition~\eqref{InterpolationU} of multi-interpolated MZVs can be obtained from
our previous Definitions~\ref{DefMuli} and~\ref{DefNewZeta1} by application of a formal map 
$V$, whose action is given as follows\footnote{Actually, $V$ can be rigorously defined algebraically.}:
\[
V\big(\zt(i_1\vt^{j_1},\dots,i_k\vt^{j_k})\big)=\zt(i_1,\dots,i_k)\cdot\prod_{\ell=1}^{k}u_{j_\ell}.
\]
This implies that
\[
\zt^{(\vec{u})}(i_1,\dots,i_k)=V\big(\ztm(i_1,\dots,i_k)\big).
\]
However, in general 
\[
V\big(\ztm(i_1,\dots,i_k)\cdot \ztm(j_1,\dots,j_m)\big)\neq V\big(\ztm(i_1,\dots,i_k) \big) \cdot V\big(\ztm(j_1,\dots,j_m)\big).
\]
It remains to be seen if there exists a suitably defined harmonic product $\ast^{(\vec{u})}$, or even a more general product. 



\subsection{Further variations}
Another generalization of $\ztt(k_1,\dots,k_n)$ are interpolated Schur multiple zeta values, as introduced in~\cite{B} (see also~\cite{NPY} for Schur multiple zeta values).
We note that it is possible to introduce multi-interpolated Schur multiple zeta values, unifying interpolated Schur multiple zeta values and multi-interpolated multiple zeta values:
the parameters $v(\mathbf{m})$, counting the vertical equalities, and $h(\mathbf{m})$, counting the horizontal equalities can be refined by taking into account the values of the equal entries,
leading to $\vec{t}^{\vec{v}(\mathbf{m})}(\vec{1}-\vec{t})^{\vec{h}(\mathbf{m})}$ (see~\cite{B}, Definition 2.3).

\smallskip 

Finally, following Ohno and Wayama's~\cite{Ohno} generalization of the Arakawa-Kaneko multiple zeta function~\cite{Arakawa},
we note that a multi-interpolated Arakawa-Kaneko multiple zeta function can be defined as follows:
\[
\xi^{\vec{t}}(k_1,\dots,k_r;s)=\frac{1}{\Gamma(s)}\int_0^{\infty} z^{s-1}\frac{\text{Li}^{\vec{t}}_{k_1,\dots,k_r}(1-e^{-z})}{e^z-1}dz,
\]
were $\text{Li}^{\vec{t}}_{k_1,\dots,k_r}(z)$ denotes a multi-interpolated multi-polylogarithm defined by
\[
\text{Li}^{\vec{t}}_{k_1,\dots,k_r}(z)
=\sum_{\ell_1\ge \cdots\ge \ell_r\ge 1}\frac{z^{\ell_1}\cdot \vec{t}^{\vec{\sigma}(\vec{\ell})}}{\ell_1^{k_1}\dots \ell_r^{k_r}}.
\]
By setting $\vec{t}=t$ we reobtain the $t$-Arakawa–Kaneko multiple zeta functions of~\cite{Ohno}.

\smallskip

It remains to be seen if multi-interpolated objects, like multi-interpolated Schur multiple zeta values 
or multi-interpolated Arakawa-Kaneko multiple zeta functions, still have interesting structural properties.

\subsection{Summary}
In this note we introduced a multi-interpolated multiple zeta value $\ztm(\vec{i})$ with variables $\vec{t}=(t_1,t_2,\dots)$, generalized the ordinary interpolated multiple zeta value $\ztt(\vec{i})$, case $t_\ell=t$, $1\le \ell$. 
A few properties of $\ztm(\vec{i})$ where established in this note, in particular, formulas for $\zt^{\vec{t}}(\{s\}_k)$, as well as a harmonic product $\mstuff$ such that 
$\ztm((i_1,\dots,i_m) \mstuff (j_1,\dots,j_k))=\ztm(i_1,\dots,i_m)\cdot \ztm(j_1,\dots,j_k)$. We proposed several open problems and pointed out different variants of multi-interpolated MZVs.

\end{document}